\documentclass[12pt,longbibliography]{article}

\usepackage[utf8]{inputenc} % Allow direct use of accents such as á é ñ.
\usepackage[T1]{fontenc}

\usepackage{graphicx}
\usepackage{amsmath}
\usepackage{amsfonts}
\usepackage{mathrsfs}           % \mathscr font.
\usepackage{amsthm}
\usepackage{amssymb}
\usepackage{color}
\usepackage{enumerate}
\usepackage{comment}

\usepackage[a4paper, margin=2.7cm]{geometry}

\usepackage[colorlinks=true,linkcolor=blue,citecolor=blue,urlcolor=blue,breaklinks]{hyperref}

\usepackage{bbm}

\usepackage{breakurl}
\usepackage{url}
\def\N{\mathbb{N}}
\def\R{\mathbb{R}}

\def\d{\,\mathrm{d}}

\newtheorem{thm}{Theorem}[section]
\newtheorem{cor}[thm]{Corollary}
\newtheorem{lem}[thm]{Lemma}

\theoremstyle{definition}
\newtheorem{dfn}[thm]{Definition}
\theoremstyle{remark}
\newtheorem{rem}[thm]{Remark}
\theoremstyle{example}

\title{Intrinsic Ultracontractivity for a class of Schrödinger Semigroups in $\mathrm{L}^{2}\left( \R^{n} \right)$
using Log-Sobolev-inequalities and duality arguments}

\author{Christoph Schwerdt$^1$ \and Ilham Ouelddris$^2$}
\date{%
	$^1$ Institute of Mathematics, University of Rostock,\\
	Ulmenstra\ss e 69, 18 057 Rostock, Germany; \\%
	\textit{E-mail address}: \texttt{\href{mailto:christoph.schwerdt@uni-rostock.de}{christoph.schwerdt@uni-rostock.de}}\\
	\ \\
	$^2$ Department of Mathmatics, \\
	Cadi Ayyad University (Morocco), \\
	Marrakesh, Morocco; \\%
    \ \\
    \today
}

\begin{document}

\maketitle

\begin{abstract}
We present a class of potentials $q \colon \R^{n} \to (0,\infty)$ that implies the \textit{weighted Schrödinger 
semigroup} $\varphi^{-1}\mathrm{e}^{-tH}\varphi$ to map a \textit{weighted Lebesgue function space} 
$\mathrm{L}_{\mu}^{1}(\R^{n})$ into a weighted Lebesgue function space $\mathrm{L}_{\mu}^{2}(\R^{n})$ 
continously at every time $t>0$ by \textit{Logarithmic Sobolev inequalities} for $H=-\Delta + q(x)$ with it's strictly 
positive ground state $\varphi \colon \R^{n} \to (0,\infty)$. We use the self-adjointness of $\mathrm{e}^{-tH}$ in 
$\mathrm{L}^{2}(\R^{n})$ to infer an \textit{intrinsic ultracontractivity}, i.e.
$$
\forall t>0 \ \exists C_{t} > 0 \ : \ \left| \mathrm{e}^{-tH} u (x) \right| \ \leq \ C_{t} \varphi(x) \| u \|_{2}
$$
for every $u \in \mathrm{L}^{2}\left( \R^{n} \right)$ almost everywhere in $\R^{n}$. 
\end{abstract}

\newpage

\tableofcontents

\section{Introduction}

\subsection{Intrinsic ultracontractivity in $\mathrm{L}^{2}(\R^{n})$}

The Hamiltonian corresponding to a quantum system in $\R^{n}$ is given by the formal operator
$H =  -\Delta + q(x)$ in $\mathrm{L}^{2}(\R^{n})$ for $n \geq 3$ dimensions with $q \in \mathrm{L}^{1}_{loc}\left( \R^{n} \right)$. 
We focus on potentials $q$ being continuous and non-negative in $\R^{n}$ such that there exists a unique strictly positive 
eigenfunction $\varphi$ corresponding to the lowest eigenvalue $E_{0}$ of $H$ which is called the system's \textit{ground state}.
We present a growth condition of $q$ that implies \textit{Rosen inequalities} of $\varphi$, i.e.
\begin{equation}
\forall \varepsilon > 0 \ \exists \gamma(\varepsilon) > 0 \ : \ - \ln \left( \varphi(x) \right) \ \leq \ \varepsilon q(x) + \gamma(\varepsilon)
\end{equation}
for almost every $x \in \R^{n}$. While these inequalities are not particularly interesting in themselves, they imply 
\textit{Logarithmic Sobolev inequalties} which are essential to prove an
\textit{intrinsic ultracontractivity} of the Schrödinger semigroup $\mathrm{e}^{-tH}$, i.e.
\begin{equation}
\forall t > 0 \ \exists C_{t} > 0 \ : \ \left| \mathrm{e}^{-tH}u (x)  \right| \ \leq \ C_{t} \varphi(x) \| u \|_{2}
\end{equation}
for every $u \in \mathrm{L}^{2}(\R^{n})$ almost everywhere in $\R^{n}$. Notive that at every time $t>0$ the asymptotic behaviour 
of $\mathrm{e}^{-tH}u$ is dominated by the ground state $\varphi$ in this case. 
In particular, for a normed eigenfunction $v$ of $H$ we conclude to 
\begin{equation}
\exists C_{\lambda} > 0 \ : \ |v(x)| \leq C_{\lambda} \varphi(x)
\end{equation}
almost everywhere in $\R^{n}$ where $\lambda$ is the associated eigenvalue of $v$. Remember that in quantum physics
eigenfunctions are characterized as probability density functions of an electron's position at a certain energy level 
$\lambda$ of the system $H$.
But apart from quantum mechanics we can interpret intrinsic ultracontractivity as perturbation result of the generator $-\Delta$ 
of the free Schrödinger semigroup given by
\begin{equation}
\mathrm{e}^{t\Delta}u (x) = \frac{1}{(4 \pi t)^{\frac{n}{2}}} 
\int_{\R^{n}} \mathrm{e}^{-\frac{|x-y|^{2}}{4t}} u(y) \ \mathrm{d}y
\end{equation} 
for $u \in \mathrm{L}^{2}(\R^{n})$. Due to the Gaussian integral kernel $\mathrm{e}^{t\Delta}$ is 
a contraction in $\mathrm{L}^{p}(\R^{n})$ for every $p \in [1,\infty]$ and further maps $\mathrm{L}^{p}(\R^{n})$ into 
$\mathrm{L}^{q}(\R^{n})$ continuously for $1 \leq p < q \leq \infty$.
For the case of $H=-\Delta + q(x)$ we cite the following characterization of 
intrinsic ultracontractivity as Lemma 4.2.2 from page 110 in \cite{Davies07}. \\

\begin{lem} The Schrödinger semigroup $\mathrm{e}^{-tH}$ is intrinsic ultracontractive if and only if both of the following 
conditions are satisfied for every time $t > 0$
\begin{enumerate}[i)]
\item $\mathrm{e}^{-tH}u (x) = \int_{\R^{n}} \ k(t,x,y) u(y) \d y$ almost everywhere in $\R^{n}$ for 
$u \in \mathrm{L}^{2}(\R^{n})$ 
\item $\exists C_{t} > 0 \ : \ 0 \leq k(t,x,y) < C_{t} \ \varphi(x) \ \varphi(y)$ almost everywhere in $\R^{2n}$\\
\end{enumerate}
\end{lem}

For potentials $q$ satisfying $q(x) > |x|^{2}$ for almost every $x \in \R^{n}$ Corollary 
\ref{Asymptotical_behaviour_varphi} implies $\varphi(x) \leq K\mathrm{e}^{-|x|}$ for a constant $K >0$ and $x \in \R^{n}$. 
Therefore the integral kernel of $\mathrm{e}^{-tH}$ indeed shares some mapping simularity to $\mathrm{e}^{t\Delta}$ 
for the case of intrinsic ultracontractivity.

\subsection{Historical context}

First results of intrinsic hypercontractivity and intrinsic ultracontractivity of Schrödinger semigroups date back to the 1970s
as mentioned on page 336 in \cite{DaviesSimon84}. However until the mid-1980s 
there was a folk belief that intrinsic ultracontractivity of Schrödinger semigroups in $\mathrm{L}^{2}(\R^{n})$ would not occur. 
A simple example in that regard is the quantum harmonic oscillator $H=-\Delta + |x|^{2}$ in $\mathrm{L}^{2}(\R^{n})$ where it 
is very easy to see that $\mathrm{e}^{-tH}$ is not intrinsic ultracontractive by 
comparing the asymptotical behaviour of eigenfunctions to the ground state. 
However, in \cite{DaviesSimon84}  intrinsic ultracontractivity of $\mathrm{e}^{-tH}$ is shown
for potentials close to $|x|^{2}$ like in
$H=-\Delta + |x|^{\beta}$ or in $H=-\Delta + |x|^{2} \left( \ln(|x|+2)\right)^{\beta}$ for any $\beta > 2$. 
Even more examples are given in \cite{Davies07} which provides 
an in-depth study on intrinsic ultracontractivity of Schrödinger semigroups in $\mathrm{L}^{2}(\R^{n})$.
For bounded domains $D \subset \R^{n}$ with a sufficiently smooth boundary Banuelos  \cite{BANUELOS1991} proves 
intrinsic ultracontractivity of Schrödinger semigroups in $\mathrm{L}^{2}(D)$ in 1990. 
In the 2000s B. Alziary and P. Takáč returned to the original studies of intrinsic ultracontractivity 
in $\mathrm{L}^{2}(\R^{n})$ in \cite{AlziaryTakac09} which was the starting point of \cite{Schwerdt_Mill_Hundertmark_2025} 
where the intrinsic ultracontractivity of $\mathrm{e}^{-tH}$ in $\mathrm{L}^{2}(\R^{n})$ is shown for potentials 
$q \colon \R^{n} \to [0,\infty)$ with
\begin{equation}
d \left( \int_{0}^{|x|} Q(t)^{\frac{1}{2}} \d t \right) \ f_{k,m} \left[ \ln \left( \int_{0}^{|x|} Q(t)^{\frac{1}{2}} \d t  \right) \right]
\leq q(x) \leq Q\left( |x| \right)
\end{equation}
for large $|x|$, $d \in (0,1]$ and auxiliary functions
$$
f_{k, m} (t) = \left( \ln^{(m)}(t) \right)^{k} \ \prod_{p=0}^{m-1} \ln^{(p)}(t)
$$ 
for $k>1$ and $m \in \N$ and large $t$.\footnote{Please mind that $\ln^{(m)}$ 
is used for the m-th iteration of the logarithmic function, i.e.
$$
\ln^{(m)}(r) = \underbrace{\ln( \dots ( \ln (r)) \dots )}_{m\text{-times}}
$$} 
Furthermore the upper boundary function $Q \colon [0,\infty) \to (0,\infty)$ of $q$ satisfies  
$
Q^{\prime}(r)/Q(r)^{-3/2} \ \to \ 0
$
for $r \to \infty$. Examples of such $Q$ are given in Lemma 3.3 by 
$$|x|^{\alpha}, |x|^{2} \left( \ln |x| \right)^{\alpha}, |x|^{2}  \left( \ln |x| \right)^{2} \left( \ln \ln |x| \right)^{\alpha}
$$ 
et cetera for $\alpha > 2$.

\subsection{Motivation of this article}

This article aims to reduce the growth conditions on the potential required for intrinsic ultracontractivity while clarifying the analytical mechanism underlying this property.\\
\ \\
We build upon the recent results of \cite{Schwerdt_Mill_Hundertmark_2025}, where intrinsic ultracontractivity of the Schr\"odinger semigroup $\mathrm{e}^{-tH}$ on $\mathrm{L}^{2}(\mathbb{R}^{n})$ was established under growth conditions involving a product of iterated logarithms. In contrast, we show that these product structures are a consequence of the proof method and not of the underlying setting.\\
\ \\
More precisely, we prove intrinsic ultracontractivity for Schr\"odinger operators 
\newline $H=-\Delta+q(x)$ with non-negative potentials satisfying
\begin{equation}
d \left( \int_{0}^{|x|} Q(t)^{\frac{1}{2}} \, \mathrm{d}t \right)
\left( \ln^{(m)} \left( \int_{0}^{|x|} Q(t)^{\frac{1}{2}} \, \mathrm{d}t \right) \right)^{k}
\leq q(x) \leq Q(|x|),
\end{equation}
for sufficiently large $|x|$ where $d \in (0,1]$, $k>0$, $m\in\mathbb{N}$, and $Q$ is an auxiliary function satisfying 
$
Q^{\prime}(r)/Q(r)^{-3/2} \ \to \ 0
$
for $r \to \infty$. Compared to \cite{Schwerdt_Mill_Hundertmark_2025} the lower boundary on $q$ is lower and significantly simpler, 
where the product 
$$
\displaystyle{\prod_{p=0}^{m-1}\ln^{(p)}(t)}
$$ 
appearing in earlier works is no longer required, and, in particular, the parameter $k\in(0,1)$ becomes admissible.\\
\ \\
The key insight of this paper is that this simplification is not accidental. It stems from a different analytical viewpoint on intrinsic ultracontractivity. Rather than following the classical route, suggested for instance in \cite{DaviesSimon84}, we include a 
duality-based approach. Weighted Lebesgue spaces $\mathrm{L}^{p}_{\mu}(\R^{n})$ and weighted Schrödinger semigroups
$$
\mathrm{e}^{-t\tilde{H}} = \frac{1}{\varphi} \mathrm{e}^{-tH} \varphi
$$
are still introduced. However, we show that our class of potentials $q$ imply the weighted Schrödinger semigroup operators 
to map continuously from $\mathrm{L}^{1}_{\mu}(\mathbb{R}^{n})$ into $\mathrm{L}^{2}_{\mu}(\mathbb{R}^{n})$. 
Then the self-adjointness of $\mathrm{e}^{-tH}$ in $\mathrm{L}^{2}(\mathbb{R}^{n})$ gives an intrinsic ultracontractivity 
in $\mathrm{L}^{2}(\mathbb{R}^{n})$.\\
\ \\
In this framework, the transition from $\mathrm{L}^{1}_{\mu}$ to $\mathrm{L}^{2}_{\mu}$ captures the essential regularizing effect of the semigroup, while the stronger $\mathrm{L}^{2}_{\mu} \to \mathrm{L}^{\infty}$ bound follows naturally by duality.
While logarithmic Sobolev inequalities still play a role in the analysis, the present approach avoids several technical layers inherent to the classical method. As a result, the proof becomes both shorter and more transparent, and it reveals that the restrictive logarithmic product conditions previously imposed are not a fundamental feature of intrinsic ultracontractivity, but rather an artifact of the $\mathrm{L}^{2} \to \mathrm{L}^{\infty}$ strategy.\\

\section{Definitions and preparations} \label{Definition_h_H}

We consider Schrödinger operators $H$ and ground states $\varphi$ 
as described in Section 2 and Schrödinger semigroups $\mathrm{e}^{-tH}$ in $\mathrm{L}^{2}(\R^{n})$ as in Section 5 of \cite{Schwerdt_Mill_Hundertmark_2025}.
In particular, we use weighted Schrödinger semigroups 
\begin{equation}
\left( \mathrm{e}^{-t\tilde{H}}u \right) (x)  = \frac{1}{\varphi(x)} \ \mathrm{e}^{-tH}(\varphi u) (x)
\end{equation}
in weighted Lebesgue spaces $\mathrm{L}^{p}_{\mu}(\R^{n})$ for $p \in [1,\infty)$ which are also found in Section 5 of \cite{Schwerdt_Mill_Hundertmark_2025}. Our goal of this article is a specific growth condition on the potential 
$q \colon \R^{n} \to [0,\infty)$ of $H = -\Delta + q(x)$
that implies an intrinsic ultracontractivity of the Schrödinger semigroup $\mathrm{e}^{-tH}$ in $\mathrm{L}^{2}(\R^{n})$.\\

Let $Q \colon [0,\infty) \to (0,\infty)$ be monotone increasing with $r^{2} < Q(r)$ and let
$k > 0, m \in \N$ and $R_{m} > 0$ be a radius such that 
\begin{enumerate}[i.)]
\item $Q$ is differentiable in $(R_{m},\infty)$, 
\item $Q^{\prime}(r)Q(r)^{-\frac{3}{2}} \to 0$ for $r \to \infty$ and
\item $\forall r > R_{m} \ : \ r^{2} < \left( \int_{0}^{r} Q(t)^{\frac{1}{2}} \d t \right) 
\left( \ln^{(m)} \left(  \int_{0}^{r} Q(t)^{\frac{1}{2}} \d t   \right) \right)^{k} < Q(r)$.\\
\end{enumerate}

\begin{lem}\label{Rosen_inequalities}
Let $Q$ be as described above and let $q \colon \R^{n} \to [0,\infty)$ be continuous and
satisfy
$$
d \left( \int_{0}^{|x|} Q(t)^{\frac{1}{2}} \d t \right) 
\left( \ln^{(m)} \left(  \int_{0}^{|x|} Q(t)^{\frac{1}{2}} \d t   \right) \right)^{k} \ \leq \ q(x) \ \leq \ Q(|x|) 
$$
for $d \in (0,1)$ and every $|x| \geq R_{m}$. Then $\varphi$ satisfies Rosen inequalities, i.e.
\begin{equation}\label{Rosen_inequ}
\forall \varepsilon > 0 \ \exists \gamma(\varepsilon) > 0 \ : \ -\ln\left( \varphi(x) \right) \ \leq \ \varepsilon q(x) + \gamma(\varepsilon)
\end{equation}
for almost every $x \in \R^{n}$.\\
\end{lem}

\begin{rem} \ \\
\begin{enumerate}[i.)]
\item Notice that the lower boundary functions of $q$ in Lemma \ref{Rosen_inequalities} are much smaller and 
simpler than in \cite{Schwerdt_Mill_Hundertmark_2025} while still being sufficient in terms of 
intrinsic ultracontractivity of $\mathrm{e}^{-tH}$. Hence all examples of $Q$ in Lemma 3.7 in 
\cite{Schwerdt_Mill_Hundertmark_2025} satisfy Lemma \ref{Rosen_inequalities} as well.
\item Lemma \ref{Rosen_inequalities} includes auxiliary potentials $Q$ much closer to $r^{2}$ such as
$$
Q(r) = r^{2} \left( \ln^{(n)} (r) \right)^{l}
$$ for $n \in \N, l > 0$ and $r$ sufficiently large. By simple calculations we 
infer that $Q^{\prime}(r)Q(r)^{-\frac{3}{2}} \to 0$ is true for $r \to \infty$. 
Furthermore, notice that
$$
\lim_{r \to \infty} \frac{r^{2}}{\int_{0}^{r} Q(t)^{\frac{1}{2}} \d t} = \lim_{r \to \infty} \frac{2r}{r \left( \ln^{(n)} (r) \right)^{l/2}} = 0
$$
holds by L'Hospital. Therefore $r^{2} < \int_{0}^{|x|} Q(t)^{\frac{1}{2}} \d t$ for large $r$ is implied. By the monotonicity of $Q$
we argue that
$$
\frac{\int_{0}^{r} Q(t)^{\frac{1}{2}} \d t}{Q(r)} \leq \frac{rQ(r)^{1/2}}{Q(r)} = \frac{r}{Q(r)^{1/2}}
$$
and 
$\left( \ln^{(m)} \left(  \int_{0}^{r} Q(t)^{\frac{1}{2}} \d t   \right) \right)^{k}  \leq 
\left( \ln^{(m)} \left( rQ(r)^{\frac{1}{2}} \right)   \right)^{k}$
are true. Therefore we conclude that
$$
\frac{\int_{0}^{r} Q(t)^{\frac{1}{2}} \d t}{Q(r)} \left( \ln^{(m)} \left(  \int_{0}^{r} Q(t)^{\frac{1}{2}} \d t   \right) \right)^{k}
$$
converges to $0$ for $r \to \infty$ and the choice of $k \in (0,l/2)$ and $m \geq n$ since
\begin{align*}
\frac{\int_{0}^{r} Q(t)^{\frac{1}{2}} \d t}{Q(r)} \left( \ln^{(m)} \left(  \int_{0}^{r} Q(t)^{\frac{1}{2}} \d t   \right) \right)^{k}
& \leq \frac{r}{Q(r)^{1/2}} \left( \ln^{(m)} \left( rQ(r)^{\frac{1}{2}} \right)   \right)^{k}
\leq \frac{\left( \ln^{(m)} \left( r^{2+l/2} \right)   \right)^{k} }{ \left( \ln^{(n)} (r) \right)^{l/2}} \\
& \leq \frac{ 2^{k} \left( \ln^{(m)} \left( r \right)   \right)^{k} }{ \left( \ln^{(n)} (r) \right)^{l/2}}
 \leq  2^{k} \left( \ln^{(n)} \left( r \right)   \right)^{k-l/2}
\end{align*}
is true for large $r$. Mind that we used 
$\ln^{(m)} \left( r^{2+ l/2} \right) \leq 2 \ln^{(m)}(r)$ for large $r$ which we have to justify. 
Notice that
$$
\ln^{(2)} \left( r^{2+l/2} \right) = \ln \left( 2 + l/2 \right) + \ln^{(2)}(r) \leq 2 \ln^{(2)}(r)
$$ 
implies $\ln^{(3)} \left( r^{2+l/2} \right) \leq \ln(2) + \ln^{(3)}(r) \leq 2 \ln^{(3)}(r)$ for large $r$ by 
using the monotonicity of the logarithm on $(0,\infty)$. Iteratively 
$\ln^{(m)} \left( r^{2+ l/2} \right) \leq 2 \ln^{(m)}(r)$ is implied for large $r$. Hence we also have shown that
$$
\left( \int_{0}^{r} Q(t)^{\frac{1}{2}} \d t \right) \left( \ln^{(m)} \left(  \int_{0}^{r} Q(t)^{\frac{1}{2}} \d t   \right) \right)^{k} < Q(r)
$$
is true for large $r$, $k \in (0,l/2)$ and $m \geq n$.\\
\end{enumerate}
\end{rem}

\begin{proof}[Proof of Lemma \ref{Rosen_inequalities}]
Using Subsection 4.1 in \cite{Schwerdt_Mill_Hundertmark_2025} on 
$\psi(x) = \exp\left( -\sqrt{2}\int_{0}^{|x|} Q(t)^{\frac{1}{2}} \d t \right)$ for
$x \in \R^{n}$, the inequality
\begin{equation}
-\ln\left( \varphi(x) \right) \ \leq \ \sqrt{2}\int_{0}^{|x|} Q(t)^{\frac{1}{2}} \d t + C
\end{equation}
holds for every $x \in \R^{n}$ where $C > 0$ is a constant. Now we use the standard version of 
Young's inequality for increasing functions to include
$\varepsilon > 0$. 
For any $m \in \N$ and $k > 0$ we define $f_{k,m} \colon [0,\infty) \to [0,\infty)$ by
$$
f_{k,m}(q)=\left\{\begin{array}{ll} 
	\left( \ln^{(m)} (q) \right)^{k}, & r_{m} \leq q \\
         q/r_{m}, & 0 \leq q < r_{m}
\end{array}\right. 
$$
where $r_{m} = \exp^{(m)}(1)$.\footnote{Mind that 
$\exp^{(m)}(1) = \underbrace{\exp \left( \exp \left( \dots \left( \exp(1) \right) \dots \right)\right)}_{m-times}$.}
Let $g_{k,m}$ be the inverse function of $f_{k,m}$ on $[0,\infty)$. Then 
$$
ab \ \leq \ \int_{0}^{a} f_{k,m} (t) \d t + \int_{0}^{b} g_{k,m}(t)  \d t
\ \leq \ a f_{k,m} (a) + b g_{k,m} (b)
$$ 
holds for every $a,b > 0$. Define 
$a =  \int_{0}^{|x|} Q(t)^{\frac{1}{2}} \ \mathrm{d}t$ and $b = \frac{\sqrt{2}}{d \varepsilon}$ to conclude to
\begin{equation} \label{Young_Q}
\int_{0}^{|x|} Q(t)^{\frac{1}{2}} \ \mathrm{d}t \ \leq \ \frac{d \varepsilon}{\sqrt{2}} 
\left( \int_{0}^{|x|} Q(t)^{\frac{1}{2}} \ \mathrm{d}t \right)
\left( \ln^{(m)} \left( \int_{0}^{|x|} Q(t)^{\frac{1}{2}} \ \mathrm{d}t \right) \right)^{k} + g_{k,m} \left( \frac{\sqrt{2}}{d \varepsilon} \right)
\end{equation}
for $|x|$ large. Then we infer
\begin{align*}
-\ln\left( \varphi(x) \right) \ & \leq \sqrt{2}\int_{0}^{|x|} Q(t)^{\frac{1}{2}} \d t + C \\
& \leq d \varepsilon \left( \int_{0}^{|x|} Q(t)^{\frac{1}{2}} \ \mathrm{d}t \right)
\left( \ln^{(m)} \left( \int_{0}^{|x|} Q(t)^{\frac{1}{2}} \ \mathrm{d}t \right) \right)^{k} + 
\sqrt{2} \ g_{k,m} \left( \frac{\sqrt{2}}{d \varepsilon} \right) + C \\
& \leq \varepsilon q(x) + \underbrace{\sqrt{2} \ g_{k,m} \left( \frac{\sqrt{2}}{d \varepsilon} \right) + C}_{=\gamma(\varepsilon)} 
\end{align*}
for every $|x| > R_{m}$. Mind that $R_{m}$ is supposed to be sufficiently large for not having to include an additional and even bigger 
radius to the argumentation. Due to the continuity of $\varphi$ on the compact subset $\overline{B_{R_{m}}(0)}$ in $\R^{n}$
we argue that 
$$
\varphi(x) \geq \min_{y \in \overline{B_{R_{m}}(0)}} \varphi(y) = \delta_{0} > 0
$$
is true for every $x \in \overline{B_{R_{m}}(0)}$ since $\varphi(x) > 0$ for every $x \in \overline{B_{R_{m}}(0)}$.
Please compare with Lemma 2.1 and Lemma 2.2 in \cite{Schwerdt_Mill_Hundertmark_2025}.
So we choose $C$ sufficiently large that Rosen inequalities (\ref{Rosen_inequ}) are implied.\\
\end{proof}

Next, we cite Rosen's Lemma from Subsection 5.3 in \cite{Schwerdt_Mill_Hundertmark_2025} to later use for an 
argumentation on the monotonicity of $\{ p \mapsto \|\mathrm{e}^{-t\tilde{H}}u \|_{p, \mu}  \}$ for 
$u \in D( \tilde{H} ) \cap \mathrm{L}_{\mu}^{\infty}\left( \R^{n} \right)$ being non-negative almost everywhere in
$\R^{n}$. \\

\begin{lem}[Rosen's Lemma] \label{Rosen_lemma} \ \\
Let $Q$ and $q$ be as in Lemma \ref{Rosen_inequalities}. Then
for every $\varepsilon > 0$ and every $u \in D( \tilde{H} ) \cap \mathrm{L}_{\mu}^{\infty}\left( \R^{n} \right)$ 
being non-negative almost everywhere in $\R^{n}$ the Logarithmic Sobolev inequalities
\begin{align*}
& \int_{\R^{n}} \left( \mathrm{e}^{-t\tilde{H}} u(x) \right)^{p} \ \ln \left( \mathrm{e}^{-t\tilde{H}} u (x) \right) \d \mu(x) \leq \\ 
& \varepsilon \ 
\langle \tilde{H} \mathrm{e}^{-t\tilde{H}} u, (\ \mathrm{e}^{-t\tilde{H}} u \ )^{p-1} \rangle_{\mu} 
 + \frac{2\beta(\varepsilon)}{p} \ \| \mathrm{e}^{-t\tilde{H}} u \|_{p, \mu}^{p} + \| \mathrm{e}^{-t\tilde{H}} u \|_{p, \mu}^{p}  \ln  \| \mathrm{e}^{-t\tilde{H}} u \|_{p, \mu}
\end{align*}
are implied for 
$\beta(\varepsilon) = \frac{\varepsilon}{2} - \frac{n}{4} \ln( \frac{\varepsilon}{2}) + \gamma( \frac{\varepsilon}{2}) + C$
where $C \in \R$ is a constant. 
\end{lem}

\section{Intrinsic ultracontractivity of $\mathrm{e}^{-tH}$ in $\mathrm{L}^{2}(\R^{n})$} \label{Main_Theorem}

\begin{thm}\label{IU_of_e^{-tH}}
For potentials $q$ as defined in Lemma \ref{Rosen_inequalities} the Schrödinger semigroup 
$\mathrm{e}^{-tH}$ is intrinsic ultracontractive in $\mathrm{L}^{2}(\R^{n})$, i.e.
\begin{equation}
\forall t > 0 \exists C_{t} > 0 \ : \  \bigl| \mathrm{e}^{-tH}u (x) \bigr| \ \leq \ C_{t} \|u\|_{2} \ \varphi (x) \\
\end{equation}
holds almost everywhere in $\R^{n}$ for every $u \in \mathrm{L}^{2}(\R^{n})$.\\
\end{thm}

Mind that $|x|^{2} \ < \ q(x)$ holds for every $|x| \geq R$ and hence use Corollary \ref{Cor_Appendix_Agmon}.\\

\begin{cor}\label{Corollary_Main}
For potentials $q$ as in Theorem \ref{IU_of_e^{-tH}} we further conclude to
\begin{equation}
\forall t > 0 \exists C_{t} > 0 \ : \ \left| \mathrm{e}^{-tH}u (x) \right| \ \leq \ C_{t} \ \|u\|_{2} \ \mathrm{e}^{ -|x| }
\end{equation}
almost everywhere in $\R^{n}$ for every $u \in \mathrm{L}^{2} \left( \R^{n} \right)$.\\ 
\end{cor}

\section{Proof of Theorem \ref{IU_of_e^{-tH}}} \label{Proof_IU}

\subsection{$\mathrm{e}^{-t\tilde{H}} \in \mathcal{B}\left( \mathrm{L}_{\mu}^{1}(\R^{n}), \mathrm{L}_{\mu}^{2}(\R^{n}) \right)$}

Let $t > 0$ be arbitrarily but fixed. In this subsection we prove that $\mathrm{e}^{-t\tilde{H}}$ maps continuously from
the weighted Lebesgue space $\mathrm{L}_{\mu}^{1}(\R^{n})$ into $\mathrm{L}_{\mu}^{2}(\R^{n})$.\\

\begin{enumerate}[i.)] 
\item
Let $u \in D( \tilde{H} ) \cap \mathrm{L}_{\mu}^{\infty}(\R^{n})$ 
be non-negative almost everywhere in $\R^{n}$ and define $p \colon [0,t] \to [1,2]$ by 
$p(s) = 1+ s/t$. Using Lemma 2.2.2 on page 64 in \cite{Davies07} we conclude
\begin{align*}
& \frac{\d}{\d s} \|\mathrm{e}^{-s\tilde{H}}u \|^{p(s)}_{p(s), \mu} \\
& = \ p(s) \langle -\tilde{H}\mathrm{e}^{-s\tilde{H}}u, \left( \mathrm{e}^{-s\tilde{H}}u \right)^{p(s)-1}  \rangle_{\mu} 
+ p^{\prime}(s) \int_{\R^{n}} \left( \mathrm{e}^{-s\tilde{H}}u (x) \right)^{p(s)} \ln \left( \mathrm{e}^{-s\tilde{H}}u (x) \right) \d \mu(x)
\end{align*}
for every $s \in [0,t]$. Remember that $\mathrm{e}^{-s\tilde{H}}u \in \mathrm{L}_{\mu}^{\infty}(\R^{n})$
is implied for every $s \in [0,t]$ and therefore $\mathrm{e}^{-s\tilde{H}}u$ is contained in $\mathrm{L}_{\mu}^{p}(\R^{n})$
for every $p \in [1,\infty)$ since $\mu$ is a probability measure. Let us include a formally defined auxiliary function 
$N \colon [0,t] \to \R$ that we will define later with the Logarithmic Sobolev inequalities of 
Lemma \ref{Rosen_lemma} in mind. The derivative of
$$
\ln \left( \ \mathrm{e}^{-N(s)}  \| \mathrm{e}^{-s\tilde{H}} u \|_{p(s),\mu} \ \right) 
= - N(s) + \frac{1}{p(s)} \ln \| \mathrm{e}^{-s\tilde{H}} u \|_{p(s),\mu}^{p(s)}
$$
is equal to
\begin{equation*}
 - N^{\prime}(s) - \frac{p^{\prime}(s)}{p(s)} \  \ln \left\| \mathrm{e}^{-s\tilde{H}} u \right\|_{p(s),\mu}
 + \frac{1}{p(s)} \ \left\| \mathrm{e}^{-s\tilde{H}} u \right\|_{p(s),\mu}^{-p(s)} \ 
\frac{\d}{\d s} \left\| \mathrm{e}^{-s\tilde{H}} u \right\|_{p(s),\mu}^{p(s)}.
\end{equation*}
Hence, we end up with\\
\begin{align*}
& \frac{\d}{\d s} \ \ln \left( \ \mathrm{e}^{-N(s)}  \| \mathrm{e}^{-s\tilde{H}} u \|_{p(s),\mu} \ \right)  = \\
& \frac{p^{\prime}(s) }{ p(s) \left\|  \mathrm{e}^{-s\tilde{H}} u \right\|_{p(s),\mu}^{p(s)} }  
 \left\{ \ \int_{\R^{n}} (\mathrm{e}^{-s\tilde{H}} u (x) )^{p(s)} \ \ln \left( \mathrm{e}^{-s\tilde{H}} u(x) \right) \d \mu(x) 
- \frac{p(s)}{p^{\prime}(s)} \  
\left\langle \tilde{H} \mathrm{e}^{-s\tilde{H}} u, (\mathrm{e}^{-s\tilde{H}} u)^{p(s)-1} \right\rangle_{\mu} \right. \\
&  \left. - \frac{N^{\prime}(s) p(s)}{p^{\prime}(s)} \ \left\| \mathrm{e}^{-s\tilde{H}} u \right\|_{p(s),\mu}^{p(s)} 
 - \left\| \mathrm{e}^{-s\tilde{H}} u \right\|_{p(s),\mu}^{p(s)} \ 
\ln \left\| \mathrm{e}^{-s\tilde{H}} u \right\|_{p(s),\mu} \right\} \\
\end{align*}
for every $s \in [0,t]$ where $p(s)/p^{\prime}(s) = s + t > 0$. We use the 
Logarithmic Sobolev inequalities of $\tilde{H}$ in Rosen's lemma \ref{Rosen_lemma} where we use
$p(s)/p^{\prime}(s)$ as $\varepsilon$ and define 
\begin{equation}
N(s) = 2 \int_{0}^{s} \frac{\beta ( \frac{p(r)}{p^{\prime}(r)} )}{p(r)^{2}} \ p^{\prime}(r) \d r = 
2 \int_{0}^{s} \frac{\beta\left( r + t \right)}{\left( r + t \right) \left( 1 + r/t \right)} \d r 
\end{equation} 
for $s \in [0,t]$ with $\beta( r + t ) = \frac{r + t}{2} - \frac{n}{4} \ln\left( \frac{r + t}{2} \right) + 
g_{k,m}( \frac{2^{\frac{3}{2}}}{d( r + t )} ) + C$. Then 
$$
N^{\prime}(s) \ \frac{p(s)}{p^{\prime}(s)} = \frac{2 \beta ( p(s)/p^{\prime}(s) )}{p(s)}
$$
holds for every $s \in [0,t]$. Furthermore $g_{k,m}( \frac{2^{\frac{3}{2}}}{d( r + t )} )$ is bounded
for $r \in [0,t]$ since
$$
0 < g_{k,m}\left( \frac{\sqrt{2}}{dt} \right) \leq g_{k,m}\left( \frac{2^{\frac{3}{2}}}{d( r + t )} \right)
\leq g_{k,m}\left( \frac{2^{\frac{3}{2}}}{dt} \right) < \infty
$$
holds true due to $g_{k,m}$ being strictly monotone increasing on $[0,\infty)$. Therefore the integral
$$
\int_{0}^{s} \frac{\beta\left( r + t \right)}{\left( r + t \right) \left( 1+ r/t \right)} \d r 
$$
exists for every $s \in [0,t]$ which makes $N$ a well defined function. Having defined our auxiliary function 
$N$ we simply use Rosen's lemma \ref{Rosen_lemma} to conclude 
\begin{equation}
\frac{\d}{\d s} \ \ln \left( \ \mathrm{e}^{-N(s)}  \| \mathrm{e}^{-s\tilde{H}} u \|_{p(s),\mu} \ \right) \leq 0
\end{equation}
for every $s \in [0,t]$. Therefore $\left\{ s \mapsto \mathrm{e}^{-N(s)}  \| \mathrm{e}^{-s\tilde{H}} u \|_{p(s),\mu} \right\}$
is monotone decreasing on $[0,t]$ and hence
$$
\mathrm{e}^{-N(t)} \| \mathrm{e}^{-t\tilde{H}} u \|_{2,\mu} \leq 
\mathrm{e}^{-N(t)} \| \mathrm{e}^{-t\tilde{H}} u \|_{p(t),\mu} \leq 
\mathrm{e}^{-N(0)} \| u \|_{p(0),\mu} =  \| u \|_{1,\mu}
$$
is true. Let us summarize this first result of the proof. For $u \in D( \tilde{H} ) \cap \mathrm{L}_{\mu}^{\infty}(\R^{n})$ 
being non-negative almost everywhere in $\R^{n}$ we have shown 
\begin{equation}
\| \mathrm{e}^{-t\tilde{H}} u \|_{2,\mu} \ \leq \ \underbrace{\mathrm{e}^{N(t)}}_{=C_{t}}   \| u \|_{1,\mu}.
\end{equation}

\item Let $u \in \mathrm{L}_{\mu}^{\infty}(\R^{n})$ be non-negative almost everywhere in $\R^{n}$. Then 
$\varphi u \in \mathrm{L}_{\mu}^{2}(\R^{n})$ is implied and hence there exists a sequence $(v_{k}) \in C_{c}^{\infty}(\R^{n})$
such that every $v_{k}$ is non-negative almost everywhere in $\R^{n}$ and $v_{k} \to \varphi u$ holds in 
$\mathrm{L}_{\mu}^{2}(\R^{n})$ for $k \to \infty$. We define $u_{k} = v_{k}/\varphi$ for every $k \in \N$. 
Then $\varphi u_{k} = v_{k} \in C_{c}^{\infty}(\R^{n}) \subset D(H)$ implies $u_{k} \in D(\tilde{H})$ by definition and
furthermore $v_{k}/\varphi \in \mathrm{L}_{\mu}^{\infty}(\R^{n})$ is satisfied due to the compact support of $v_{k}$
in $\R^{n}$ and the fact that $\varphi$ is continuous and strictly positive in $\R^{n}$. Please see
Lemma 2.1 and Lemma 2.2 in \cite{Schwerdt_Mill_Hundertmark_2025}. Also
$$
\left\| u_{k} - u \right\|_{2,\mu}^{2} = \int_{\R^{n}} \left| \varphi(x) u_{k}(x) - \varphi(x) u(x) \right|^{2} \d x
= \int_{\R^{n}} \left| v_{k}(x) - \varphi(x) u(x) \right|^{2} \d x
$$
converges to $0$ for $k \to \infty$ which implies that $\mathrm{e}^{-t\tilde{H}} u_{k}$ converges to $\mathrm{e}^{-t\tilde{H}} u$
in $\mathrm{L}_{\mu}^{2}(\R^{n})$. Notice that $u_{k}$ also converges to $u$ in $\mathrm{L}_{\mu}^{1}(\R^{n})$ since
$\mu$ is a probability measure and hence convergence in $\mathrm{L}_{\mu}^{2}(\R^{n})$ implies convergence in 
$\mathrm{L}_{\mu}^{1}(\R^{n})$. We summarize that
\begin{equation}\label{weighted_semigroup_bounded_L^1_L^2}
\| \mathrm{e}^{-t\tilde{H}} u \|_{2,\mu} = \lim_{k \to \infty}  \| \mathrm{e}^{-t\tilde{H}} u_{k} \|_{2,\mu}
\leq C_{t} \lim_{k \to \infty} \| u_{k} \|_{1,\mu} =C_{t}  \| u \|_{1,\mu}
\end{equation}
is satisfied for every $u \in \mathrm{L}_{\mu}^{\infty}(\R^{n})$ being non-negative almost everywhere in $\R^{n}$.

\item Let $u \in \mathrm{L}_{\mu}^{1}(\R^{n})$ be non-negative almost everywhere in $\R^{n}$. Mind that 
$\mathrm{e}^{-t\tilde{H}} u \in \mathrm{L}_{\mu}^{1}(\R^{n})$ is true but we can not confirm 
$\mathrm{e}^{-t\tilde{H}} u \in \mathrm{L}_{\mu}^{2}(\R^{n})$ at this point. We define
$u_{k} = 1_{\{ u \leq k \}} u \in  \mathrm{L}_{\mu}^{\infty}(\R^{n}) \subset \mathrm{L}_{\mu}^{1}(\R^{n})$. 
Then $(u_{k})$ converges pointwise to $u$ almost everywhere in $\R^{n}$ and $|u_{k}(x) - u(x)| \leq u(x)$. Hence 
$(u_{k})$ converges to $u$ in $\mathrm{L}_{\mu}^{1}(\R^{n})$ by Lebesgue's theorem. Notice that
$\mathrm{e}^{-t\tilde{H}} u_{k} \to \mathrm{e}^{-t\tilde{H}} u$ in $\mathrm{L}_{\mu}^{1}(\R^{n})$ is implied
since $\mathrm{e}^{-t\tilde{H}}$ is a contraction in $\mathrm{L}_{\mu}^{1}(\R^{n})$. But that does not necessarily imply
convergence in $\mathrm{L}_{\mu}^{2}(\R^{n})$. Please mind that 
$\mathrm{e}^{-t\tilde{H}} u_{k} \in \mathrm{L}_{\mu}^{\infty}(\R^{n})$ is contained in  
$\mathrm{L}_{\mu}^{1}(\R^{n}) \cap  \mathrm{L}_{\mu}^{2}(\R^{n})$ such that $(\mathrm{e}^{-t\tilde{H}} u_{k} )$
is a Cauchy sequence in $\mathrm{L}_{\mu}^{2}(\R^{n})$ by (\ref{weighted_semigroup_bounded_L^1_L^2}). 
Hence the limit $h = \lim_{k \to \infty}  \mathrm{e}^{-t\tilde{H}} u_{k}$ is contained in $\mathrm{L}_{\mu}^{2}(\R^{n})$.
Then $h \in \mathrm{L}_{\mu}^{1}(\R^{n})$ is implied since $\mu$ is a probability measure. From
\begin{align*}
\| \mathrm{e}^{-t\tilde{H}} u - h \|_{1,\mu} & \leq \| \mathrm{e}^{-t\tilde{H}} u - \mathrm{e}^{-t\tilde{H}} u_{k} \|_{1,\mu}
+ \| \mathrm{e}^{-t\tilde{H}} u_{k} - h \|_{1,\mu} \\
& \leq  \| \mathrm{e}^{-t\tilde{H}} u - \mathrm{e}^{-t\tilde{H}} u_{k} \|_{1,\mu}
+ \| \mathrm{e}^{-t\tilde{H}} u_{k} - h \|_{2,\mu} 
\end{align*} 
we conclude $\mathrm{e}^{-t\tilde{H}} u = h$ and therefore $\mathrm{e}^{-t\tilde{H}} u \in \mathrm{L}_{\mu}^{2}(\R^{n})$
with
\begin{equation}
\| \mathrm{e}^{-t\tilde{H}} u \|_{2,\mu} = \lim_{k \to \infty}  \| \mathrm{e}^{-t\tilde{H}} u_{k} \|_{2,\mu}
\leq C_{t} \lim_{k \to \infty} \| u_{k} \|_{1,\mu} =C_{t}  \| u \|_{1,\mu}
\end{equation}
for every $u \in \mathrm{L}_{\mu}^{1}(\R^{n})$ being non-negative almost everywhere in $\R^{n}$.

\item For $u \in \mathrm{L}_{\mu}^{1}(\R^{n})$ being real-valued but not necessarily non-negative almost everywhere in $\R^{n}$
we write $u=u^{+} - u^{-}$ for $0 \leq u^{+} \in \mathrm{L}_{\mu}^{1}(\R^{n})$ and $0 \leq u^{-} \in \mathrm{L}_{\mu}^{1}(\R^{n})$
for the positive and negative part of $u$. We conclude to
$$
\| \mathrm{e}^{-t\tilde{H}} u \|_{2,\mu}  
\leq \| \mathrm{e}^{-t\tilde{H}} u^{+} \|_{2,\mu}  + \| \mathrm{e}^{-t\tilde{H}} u^{-} \|_{2,\mu} \leq 
C_{t} \left( \| u^{+} \|_{1,\mu} + \| u^{-} \|_{1,\mu} \right) = C_{t}  \| u \|_{1,\mu}.
$$

\item For a general $u \in \mathrm{L}_{\mu}^{1}(\R^{n})$ we argue similar as above by writing 
$u = \operatorname{Re} u + \mathrm{i}  \operatorname{Im} u$ with 
$\operatorname{Re} u$ and $\operatorname{Im} u$ being real-valued in $\mathrm{L}_{\mu}^{1}(\R^{n})$. Hence 
the weighted Schrödinger semigroup operators $\mathrm{e}^{-t\tilde{H}}$ map $\mathrm{L}_{\mu}^{1}(\R^{n})$ 
into $\mathrm{L}_{\mu}^{2}(\R^{n})$ continuously.

\end{enumerate}

\begin{rem}
Notice that the same arguments can be used to prove that $\mathrm{e}^{-t\tilde{H}}$ maps 
$\mathrm{L}_{\mu}^{q_{1}}(\R^{n})$ into $\mathrm{L}_{\mu}^{q_{2}}(\R^{n})$ continuously for any
$1 \leq q_{1} < q_{2} < \infty$ by using 
$$
p(s) = \frac{s}{t}q_{2} + \left( 1 -  \frac{s}{t} \right) q_{1}
$$
for $s \in [0,t]$. Then $p$ is strictly monotone increasing with $p(0)=q_{1}$ and $p(t) = q_{2}$.\\
\end{rem}

\subsection{Duality argumentation for 
$\mathrm{e}^{-t\tilde{H}} \in \mathcal{B}\left( \mathrm{L}_{\mu}^{1}(\R^{n}), \mathrm{L}_{\mu}^{2}(\R^{n}) \right)$}

The adjoint $S(t)$ of 
$\mathrm{e}^{-t\tilde{H}} \in \mathcal{B}\left( \mathrm{L}_{\mu}^{1}(\R^{n}), \mathrm{L}_{\mu}^{2}(\R^{n}) \right)$ maps 
$\mathrm{L}_{\mu}^{2}(\R^{n})$ into $\mathrm{L}_{\mu}^{\infty}(\R^{n})$ continuously with
$$
\left\| S(t) u \right\|_{\infty,\mu} \leq C_{t}   \| u \|_{2,\mu}
$$
for every $u \in \mathrm{L}_{\mu}^{2}(\R^{n})$.  We have to show that both operators 
$S(t)$ and $\mathrm{e}^{-t\tilde{H}}$ coincide in $\mathrm{L}_{\mu}^{2}(\R^{n})$.
Therefore let $u,v \in \mathrm{L}_{\mu}^{2}(\R^{n})$. Please remember that 
$\mathrm{L}_{\mu}^{2}(\R^{n})$ is contained in $\mathrm{L}_{\mu}^{1}(\R^{n})$ 
and
$$
S(t) v \in \mathrm{L}_{\mu}^{\infty}(\R^{n}) \subset \mathrm{L}_{\mu}^{2}(\R^{n})
$$
because of $\mu$ being a probability measure. Therefore we argue that
\begin{equation}
\langle u, \mathrm{e}^{-t\tilde{H}}v \rangle_{\mu} = \langle \mathrm{e}^{-t\tilde{H}} u, v \rangle_{\mu}
= \langle u, S(t) v \rangle_{\mu}
\end{equation}
holds where we used the self-adjointness of $\mathrm{e}^{-t\tilde{H}}$ in $\mathrm{L}_{\mu}^{2}(\R^{n})$ for the 
first equation and understand $u$ as an element of $\mathrm{L}_{\mu}^{1}(\R^{n})$ in the second equation.
The operators $S(t)$ and $\mathrm{e}^{-t\tilde{H}}$ coincide in $\mathrm{L}_{\mu}^{2}(\R^{n})$
which implies
\begin{equation}
\left| \mathrm{e}^{-tH} (\varphi u) (x) \right| \ \leq \ C_{t} \varphi(x) \| \varphi u \|_{2} 
\end{equation}
for every $u \in \mathrm{L}_{\mu}^{2}(\R^{n})$ by the definition of  $\mathrm{e}^{-t\tilde{H}}$. Finally setting
$u = v/\varphi$ for $v \in \mathrm{L}^{2}(\R^{n})$ proves Theorem \ref{IU_of_e^{-tH}}.\\

\section{Conclusion}

In this article, we have established the intrinsic ultracontractivity of the Schr\"odinger semigroup
$\mathrm{e}^{-tH}$ for a broad class of non-negative potentials $q$ on $\mathrm{L}^{2}(\mathbb{R}^{n})$.\\
\ \\
This present work presents a new proof  for the intrinsic ultracontractivity of $\mathrm{e}^{-tH}$
based on a duality argument which offers even better results than \cite{Schwerdt_Mill_Hundertmark_2025}. 
In particular, we show that the weighted Schr\"odinger semigroup
map continuously from $\mathrm{L}^{1}_{\mu}(\mathbb{R}^{n})$ into $\mathrm{L}^{2}_{\mu}(\mathbb{R}^{n})$ 
which implies intrinsic ultracontractivity follows by self-adjointness.\\
\ \\
This methodology not only simplifies the growth assumptions imposed on the potential $q$, thereby
allowing for a wider range of parameters $k>0$, but also provides a more robust analytical framework
for investigating the contractive properties of weighted Schr\"odinger semigroups.

\newpage

\appendix

\section{Agmon's version of the comparison principle} \label{Appendix_comparison_principle}

Let $h$ denote the quadratic form introduced in Section~\ref{Definition_h_H}, and let $H$ be the associated self-adjoint operator acting on $\mathrm{L}^{2}(\mathbb{R}^{n})$.

\begin{dfn}\label{Definition_sub_supersolution}
Let $U \subseteq \mathbb{R}^{n}$ be an open set.
\begin{enumerate}[a)]
\item
The \emph{local form domain} of $h$ on $U$ is defined by
\begin{equation}
D^{U}_{\mathrm{loc}}(h)
:= \left\{ u \in \mathrm{L}^{2}_{\mathrm{loc}}(U) \;\middle|\;
\chi u \in D(h) \text{ for all } \chi \in C_{c}^{\infty}(U) \right\}.
\end{equation}
In other words, $D^{U}_{\mathrm{loc}}(h)$ consists of all functions that belong locally to the form domain of $h$.

\item
A function $u$ is called a \emph{supersolution} of $H$ in $U$ with energy $E$ if
$u \in D^{U}_{\mathrm{loc}}(h)$ and
\begin{equation}
h(u,\xi) - E \langle u,\xi \rangle \geq 0
\end{equation}
for all non-negative test functions $\xi \in C_{c}^{\infty}(U)$.
Here, the expression $h(u,\xi)$ is understood in the weak sense as
\[
h(u,\xi) = \langle \nabla u, \nabla \xi \rangle + \langle q(x)u, \xi \rangle,
\]
even if $u \notin D(h)$.

\item
A function $u$ is called a \emph{subsolution} of $H$ in $U$ with energy $E$ if
$u \in D^{U}_{\mathrm{loc}}(h)$ and
\begin{equation}
h(u,\xi) - E \langle u,\xi \rangle \leq 0
\end{equation}
for all non-negative test functions $\xi \in C_{c}^{\infty}(U)$.
\end{enumerate}
\end{dfn}

The following result is a variant of the comparison principle due to Agmon.
It corresponds to Theorem~2.7 in \cite{Hundertmark_Jex_Lange_2023}.

\begin{thm}[Agmon-type comparison principle]\label{comparision_principle}
Let $R>0$ and let $u$ be a strictly positive supersolution of $H$ with energy $E$
in a neighborhood of infinity, that is, in
\[
\Omega_R := \{ x \in \mathbb{R}^{n} \mid |x| > R \}.
\]
Assume that $v$ is a subsolution of $H$ in $\Omega_R$ with the same energy $E$ and that
\begin{equation}\label{L^2_property_Agmon}
\liminf_{N \to \infty}
\left(
\frac{1}{N^{2}}
\int_{N \le |x| \le \alpha N} |v(x)|^{2} \, \mathrm{d}x
\right) = 0
\end{equation}
for some $\alpha > 1$.

If there exist constants $\delta>0$ and $C \ge 0$ such that
\begin{equation}\label{boundary_property_Agmon}
v(x) \le C u(x)
\end{equation}
for almost every $x$ with $R < |x| \le R+\delta$, then the inequality
\[
v(x) \le C u(x)
\]
holds for almost every $x \in \Omega_R$.
\end{thm}

\begin{rem} \label{Remark_comparision_principle} \ 
\begin{enumerate}[i.)]
\item In its original form  \cite{Agmon85}
Agmon assumes that the supersolution $u$ as well as the subsolution $v$ are both continuous in $\overline{\Omega_{R}}$. 
However, this additional assumption of $u$ and $v$ being continuous are only made to guarantee that
condition (\ref{boundary_property_Agmon}) holds for a constant $C = c_{2}/c_{1}$ with 
$c_{1} = \inf_{R \leq |x| \leq R+ \delta} u(x)$ and $c_{2} = \sup_{R \leq |x| \leq R+ \delta} \left| v(x) \right|$ and arbitrary $\delta > 0$.
Please see Remark 2.9 on page 12 in \cite{Hundertmark_Jex_Lange_2023} for further details.
\item Furthermore, for $v \in \mathrm{L}^{2}\left( \Omega_{R} \right)$ condition (\ref{L^2_property_Agmon}) is already implied. \\
\end{enumerate}
\end{rem}

Please notice that we consider Schrödinger potentials $q$ such that $|x|^{2} < q(x)$ holds for every $R < |x|$. Using  Theorem \ref{comparision_principle} we are able to determine the asymptotic behaviour of
the ground state $\varphi$ of $H$.\\

\begin{cor}\label{Asymptotical_behaviour_varphi}
Let $q$ satisfy $|x|^{2} \leq q(x)$ for $R < |x|$ and a given radius $R > 0$. Furthermore, let 
$E_{0} = \min \left( \sigma(H) \right) \geq 0$ be the ground state energy of $H$ with a continuous and 
strictly positive ground state $\varphi$. Then there exists a constant $C > 0$ such that
\begin{equation}
\varphi(x) \ \leq \ C \mathrm{e}^{ -|x| }
\end{equation} 
holds for every $x \in \R^{n}$.
\end{cor}

\begin{proof}
\begin{enumerate}[i.)]
\item We use Theorem \ref{comparision_principle}. Therefore, please mind that 
$\varphi \in D(H) \subseteq \mathrm{L}^{2} \left( \R^{n} \right)$ is a subsolution of $H$ with the energy $E_{0}$
in $\Omega_{r}$ for any radius $r > 0$. We will use $|x|^{2} \leq q(x)$ for $R < |x|$ 
to define a positive and continuous supersolution $u \colon \R^{n} \to \R$ of $H$ with $E_{0}$ in $\Omega_{\tilde{R}}$ 
for a radius $\tilde{R} > R$. 
Then $\varphi$ satisfies condition (\ref{L^2_property_Agmon}) since it is contained in $\mathrm{L}^{2} \left( \R^{n} \right)$.
Furthermore, due to the continuity of $\varphi$ and $u$ condtion (\ref{boundary_property_Agmon}) is true by 
Remark \ref{Remark_comparision_principle}.
So there exists a constant $K \geq 0$ with
$$
\varphi(x) \ \leq \ K u(x)
$$
for every $x \in \Omega_{\tilde{R}}$.
\item We are concerned with solutions to the following eigenvalue inequality
\begin{equation} \label{Supersolution_H}
0 \ \leq \ \left( -\Delta + |x|^{2} \right)u(x) - E_{0}u(x)
\end{equation}
for a radius  $\tilde{R} > 0$ and $|x| \geq \tilde{R}$. The idea now is to find a $\beta \in \R$ such that 
$u(x) = \exp \left(- |x|^{\beta} \right)$ is a solution to (\ref{Supersolution_H}).
With a slight abuse of notation let us write $u(r) =  \exp \left(- r^{\beta} \right)$
and consider the radial Schrödinger inequality
\begin{align*}
0 \ & \leq \ - u^{\prime \prime}(r) - \frac{n-1}{r} u^{\prime}(r) + r^{2} u(r) - E_{0} u(r) \\
& = \ u(r) \left( - \beta^{2} r^{2(\beta - 1)} + \beta \left( \beta - 2 +n \right) r^{\beta - 2}  + r^{2} - E_{0} \right).
\end{align*}
We choose $\beta = 1$ such that $u(x) = \mathrm{e}^{ -|x| }$ satisfies (\ref{Supersolution_H}) for 
$|x| \geq \left( 1 + E_{0} \right)^{\frac{1}{2}}$.
\item Let $\tilde{R} = \max \left( R, \left( 1+ E_{0} \right)^{\frac{1}{2}} \right)$ and 
$\xi \in C_{c}^{\infty} \left( \Omega_{\tilde{R}} \right)$ be non-negative. Then
\begin{equation}
0 \leq \langle \left(-\Delta + |x|^{2} - E_{0} \right) u, \xi \rangle \leq \langle \left(-\Delta + q(x) - E_{0} \right) u, \xi \rangle
= h(u, \xi) - E_{0} \langle u, \xi \rangle
\end{equation}
follows and hence there exists a constant $K \geq 0$ such that $\varphi(x) \leq K \mathrm{e}^{ -|x| }$ holds for
every $x \in \Omega_{\tilde{R}}$ by Theorem \ref{comparision_principle}. Mind that for 
$|x| \leq \tilde{R}$ we conclude 
$$
\varphi(x) \leq \max_{|x| \leq \tilde{R}} \varphi(x) =: m \leq m\mathrm{e}^{ \tilde{R}} \ \mathrm{e}^{ -|x| }.
$$
Therefore $\varphi(x) \leq (K +  m\mathrm{e}^{ \tilde{R}}) \ \mathrm{e}^{ -|x| }$ holds for every $x \in \R^{n}$.
\end{enumerate}
\end{proof}

For a given intrinsic ultracontractivity of $\mathrm{e}^{-tH}$ with $|x|^{2} < q(x)$ for every $R < |x|$, the asymptotic 
behaviour of $\mathrm{e}^{-tH}u$ is determined by $\mathrm{e}^{-|x|}$ at every time $t$.\\

\begin{cor}\label{Cor_Appendix_Agmon}
Let $q$ satisfy $|x|^{2} \leq q(x)$ for $R < |x|$ and a given radius $R > 0$. Additionally, let the Schrödinger semigroup
 $\mathrm{e}^{-tH}$ be intrinsic ultracontractive. Then for every $t > 0$ there 
exists a constant $C_{t} > 0$ that satisfies
\begin{equation}
\left| \mathrm{e}^{-tH}u (x) \right| \ \leq \ C_{t} \ \mathrm{e}^{ -|x| }  \|u\|_{2}
\end{equation}
for every $u \in \mathrm{L}^{2} \left( \R^{n} \right)$ and almost every $x \in \R^{n}$.
\end{cor}

\newpage

\bibliography{references}
\bibliographystyle{plain}

\end{document}